\documentclass[12pt]{amsart}

\input xy
\xyoption{all}
\usepackage{amssymb}
\usepackage{amsmath}
\usepackage[margin=1.2in]{geometry}
\usepackage{hyperref}
\usepackage{wrapfig,lipsum,booktabs}
\DeclareMathOperator*{\colim}{colim}

\begin{document}

\newcommand{\nc}{\newcommand}

\newcommand{\colvec}[2]{\left  ( \begin{array}{cc} #1  \\
     #2  \end{array} \right ) }

\newcommand{\Tr}{\,{\rm Tr}\,}
\newcommand{\End}{\,{\rm End}\,}
\newcommand{\Hom}{\,{\rm Hom}\,}

\newcommand{\Ker}{ \,{\rm Ker} \,}

\newcommand{\bla}{\phantom{bbbbb}}
\newcommand{\onebl}{\phantom{a} }
\newcommand{\eqdef}{\;\: {\stackrel{ {\rm def} }{=} } \;\:}
\newcommand{\sign}{\: {\rm sign}\: }
\newcommand{\sgn}{ \:{\rm sgn}\:}
\newcommand{\half}{ {\frac{1}{2} } }
\newcommand{\vol}{ \,{\rm vol}\, }

% define abbreviations for most common commands
%
\newcommand{\br}[1]{\left[ #1 \right]}
\newcommand{\beq}{\begin{equation}}
\newcommand{\eeq}{\end{equation}}
\newcommand{\beqst}{\begin{equation*}}
\newcommand{\eeqst}{\end{equation*}}
\newcommand{\barr}{\begin{array}}
\newcommand{\earr}{\end{array}}
\newcommand{\beqar}{\begin{eqnarray}}
\newcommand{\eeqar}{\end{eqnarray}}
\newtheorem{theorem}{Theorem}[section]
\newtheorem{corollary}[theorem]{Corollary}
\newtheorem{lemma}[theorem]{Lemma}
\newtheorem{prop}[theorem]{Proposition}
\newtheorem{definition}[theorem]{Definition}
\newtheorem{remit}[theorem]{Remark}
\newtheorem{conjecture}[theorem]{Conjecture}
\newcommand{\matr}[4]{\left \lbrack \begin{array}{cc} #1 & #2 \\
     #3 & #4 \end{array} \right \rbrack}

\newenvironment{rem}{\begin{remit}\rm}{\end{remit}}
\newcommand{\quott}{/\!/}
\newcommand{\hkquott}{/\!/\!/\!/}

% black board bold face
%note \AA is already defined!
\newcommand{\aff}{{\mathbb A}}
\newcommand{\RR}{{\mathbb R}}
\newcommand{\CC}{{\mathbb C}}
\nc{\FF}{ {\mathbb F} } 
\newcommand{\ZZ}{{\mathbb Z}}
\newcommand{\PP}{ {\mathbb P} }
\newcommand{\QQ}{{\mathbb Q}}
\newcommand{\UU}{{\mathbb U}}

% calligraphic letters
\newcommand{\cala}{{\mbox{$\mathcal A$}}}
\newcommand{\calb}{{\mbox{$\mathcal B$}}}
\newcommand{\calc}{{\mbox{$\mathcal C$}}}
\newcommand{\cald}{{\mbox{$\mathcal D$}}}
\newcommand{\cale}{{\mbox{$\mathcal E$}}}
\newcommand{\calf}{{\mbox{$\mathcal F$}}}
\newcommand{\calg}{{\mbox{$\mathcal G$}}}
\newcommand{\calh}{{\mbox{$\mathcal H$}}}
\newcommand{\cali}{{\mbox{$\mathcal I$}}}
\newcommand{\calj}{{\mbox{$\mathcal J$}}}
\newcommand{\calk}{{\mbox{$\mathcal K$}}}
\newcommand{\call}{{\mbox{$\mathcal L$}}}
\newcommand{\calm}{{\mbox{$\mathcal M$}}}
\newcommand{\caln}{{\mbox{$\mathcal N$}}}
\newcommand{\calo}{{\mbox{$\mathcal O$}}}
\newcommand{\calp}{{\mbox{$\mathcal P$}}}
\newcommand{\calq}{{\mbox{$\mathcal Q$}}}
\newcommand{\calr}{{\mbox{$\mathcal R$}}}
\newcommand{\cals}{{\mbox{$\mathcal S$}}}
\newcommand{\calt}{{\mbox{$\mathcal T$}}}
\newcommand{\calu}{{\mbox{$\mathcal U$}}}
\newcommand{\calv}{{\mbox{$\mathcal V$}}}
\newcommand{\calw}{{\mbox{$\mathcal W$}}}
\newcommand{\calx}{{\mbox{$\mathcal X$}}}
\newcommand{\caly}{{\mbox{$\mathcal Y$}}}
\newcommand{\calz}{{\mbox{$\mathcal Z$}}}

\newcommand{\ib}{   }

%***************************

%

%

%Replace greek letters by their roman equivalents with \
%Slightly nonstandard:  theta is \t, tau is \ta, no omicron
\def\a{\alpha}
\def\b{\beta}
\def\g{\gamma}
\def\d{\delta}
\def\e{\epsilon}
\def\z{\zeta}
\def\h{\eta}
\def\t{\theta}
\def\k{\kappa}
\def\l{\lambda}
\def\m{\mu}
\def\n{\nu}
\def\x{\xi}
\def\p{\pi}
\def\r{\rho}
\def\s{\sigma}
\def\ta{\tau}
\def\u{\upsilon}
\def\ph{\phi}
\def\c{\chi}
\def\ps{\psi}
\def\o{\omega}

\def\G{\Gamma}
\def\D{\Delta}
\def\T{\Theta}
\def\L{\Lambda}
\def\X{\Xi}
\def\P{\Pi}
\def\U{\Upsilon}
\def\Ph{\Phi}
\def\Ps{\Psi}
\def\O{\Omega}

\nc{\itwopi}{ { 2 \pi i} }

%*************************

% \setlength{\textwidth}{6.7in}
%\setlength{\textheight}{7.5truein}
%\setlength{\evensidemargin}{0in}
%\setlength{\oddsidemargin}{0in}
%\setlength{\topmargin}{0truein}
%\setlength{\parskip}{0.2\baselineskip}
%

%%%%%%%%%%%%%

\newcommand{\cy}{C^\infty}
\nc{\su}{ {\frak{su} } }
\nc{\liet}{ \mathfrak{t} } 
\nc{\lieg}{\mathfrak{g} } 
%\nc{\diag}{{\rm diag}}
\nc{\Lie}{\rm Lie}

\newcommand{\Ham}{H}
\newcommand{\dhab}{\Ham_a}
\newcommand{\dhbb}{\Ham_b}
\newcommand{\dhpb}{\Ham_p}
\newcommand{\dhqb}{\Ham_q}
\newcommand{\dhub}{\Ham_u}
\newcommand{\dhvb}{\Ham_v}
\newcommand{\dhxb}{\Ham_x}
\newcommand{\dhyb}{\Ham_y}

\newcommand{\dee}{\delta}

\newcommand{\calM}{{\mathcal M}}
\newcommand{\calO}{{\mathcal O}}
\newcommand{\calE}{{\mathcal E}}
\newcommand{\calG}{{\mathcal G}}
\newcommand{\calTR}{{\mathcal TR}}

\newcommand{\quot}{\mathop{\rm quot}\nolimits}
\newcommand{\diag}{\mathop{\rm diag}\nolimits}

\newcommand{\gradH}{\nabla \Ham}

\newcommand{\ga}{T_a}
\newcommand{\gp}{T_p}
\newcommand{\gq}{T_q}
\newcommand{\gu}{T_u}
\newcommand{\gv}{T_v}
\newcommand{\gx}{T_x}
\newcommand{\gy}{T_y}
\newcommand{\gb}{T_b}

\newcommand{\hab}{d_{\ga}\Ham}
\newcommand{\hbb}{d_{\gb}\Ham}
\newcommand{\hpb}{d_{\gp}\Ham}
\newcommand{\hqb}{d_{\gq}\Ham}
\newcommand{\hub}{d_{\gu}\Ham}
\newcommand{\hvb}{d_{\gv}\Ham}
\newcommand{\hxb}{d_{\gx}\Ham}
\newcommand{\hyb}{d_{\gy}\Ham}

\newcommand{\cross}{\times}

\newcommand{\npb}{\|\gp\|^2}
\newcommand{\nqb}{\|\gq\|^2}
\newcommand{\nub}{\|\gu\|^2}
\newcommand{\nvb}{\|\gv\|^2}
\newcommand{\nxb}{\|\gx\|^2}
\newcommand{\nyb}{\|\gy\|^2}

\newcommand{\ca}{c_{a}}
\newcommand{\cb}{c_{b}}

\newcommand{\bbZ}{\mathbb{Z}}
\newcommand{\srm}[1]{\stackrel{#1}{\maps}}
\newcommand{\maps}{\longrightarrow}
\def\co{\colon\thinspace}
\newcommand{\xmaps}{\xrightarrow}
\newcommand{\MSU}{\mathcal{M}_{SU}}
\newcommand{\MU}{\mathcal{M}_{U}}
\newcommand{\injects}{\hookrightarrow}
\newcommand{\heq}{\simeq}

\title[The prequantum line bundle]{The prequantum line bundle on the moduli space of flat $SU(N)$ connections on a Riemann surface and the homotopy of the large $N$ limit. }
\author{Lisa C. Jeffrey}\thanks{LJ was partially supported by a grant from NSERC}
\address{Department of Mathematics \\
University of Toronto \\ Toronto, Ontario \\ Canada}
\email{jeffrey@math.toronto.edu}
\author{Daniel A. Ramras}\thanks{DR was partially supported by a grant from the Simons Foundation (\#279007)}
\address{Department of Mathematical Sciences \\
IUPUI \\ Indianapolis, IN 46202 \\ USA}
\email{dramras@math.iupui.edu}
\author{Jonathan Weitsman}\thanks{JW was partially supported by NSF grant DMS-12/11819}
\address{Department of Mathematics \\
Northeastern University \\ Boston, MA 02115\\ USA}
\email{j.weitsman@neu.edu}



\begin{abstract}We show that the prequantum line bundle on the moduli
space of flat $SU(2)$ connections on a closed Riemann surface of positive genus has degree 1.  It then follows from work of Lawton and the second author that the classifying map for this line bundle induces a homotopy equivalence between the stable moduli space of flat $SU(n)$ connections, in the limit as $n$ tends to infinity, and $ \CC P^\infty$. Applications to the stable moduli space of flat unitary connections are also discussed. 
\end{abstract}

\maketitle

\section{Introduction}

\newcommand{\lineb}{\mathcal{L}}
\newcommand{\calA}{\mathcal{A}}
\newcommand{\cA}{\calA}
\newcommand{\gauge}{\mathcal{G}}
Let $G$ be a simply connected compact Lie group and let $\Sigma$ be a closed oriented
2-manifold of genus $g > 0$. In \cite{RSW} Ramadas, Singer and Weitsman construct 
a line bundle $\lineb$ over the moduli space of 
gauge equivalence classes of flat connections  $\calA_F(\Sigma)/\calG$
on a trivial $G$-bundle on $\Sigma.$  This bundle possesses a natural connection, whose curvature is a scalar multiple of Goldman's symplectic form.

The purpose of this paper is to compute the degree (that is, the first Chern class)
of the 
line bundle described in \cite{RSW} in the case $G=SU(2)$.
Our main theorem is
\begin{theorem} \label{thm1} The degree of the line bundle is 1.
\end{theorem} 

As we will explain, the second integral cohomology group of the moduli space is infinite cyclic, and the Theorem implies that the first Chern class of $\lineb$ is a generator.  There is in fact a preferred generator (depending on the orientation of $\Sigma$), which agrees with $c_1 (\lineb)$.

In view of Question 5.6 of \cite{LR}, Theorem~\ref{thm1} has the following
corollary:
\begin{corollary} Let $\Sigma$ be a closed oriented 2-manifold of genus $g > 0$. Let   $\calG_n$ be the gauge group of the trivial $SU(n)$--bundle on $\Sigma$, and let $\calA_F^{SU(n)}(\Sigma)$ denote the space of flat connections on this bundle.
The classifying maps for the line bundles $\lineb\to \calA_F^{SU(n)}(\Sigma)/\calG_n$ induce a 
homotopy equivalence $\colim_{n\to \infty} \calA_F^{SU(n)}(\Sigma)/\calG_n \simeq \CC P^\infty$.
\end{corollary}

It was previously shown in \cite{LR} that the stable moduli space $\colim_{n\to \infty} \calA_F^{SU(n)}(\Sigma)/\calG_n \cong
\colim_{n\to \infty} \Hom(\pi_1 \Sigma, SU(n))/SU(n)$ is a $K(\bbZ, 2)$ space, and hence is homotopy equivalent to $\CC P^\infty$.  This corollary gives a geometric viewpoint on this homotopy equivalence.  In Section~\ref{conj}, we also obtain a geometric viewpoint on the homotopy equivalence $\colim_{n\to \infty} \Hom(\pi_1 \Sigma, U(n))/U(n) \simeq (S^1)^{2g} \cross \CC P^\infty$ from~\cite{R}.

Our computation of $c_1 (\lineb)$ in genus 1 (Section~\ref{g=1}) is similar to Kirk--Klassen~\cite[Theorem 2.1]{KK}.\footnote{Kirk and Klassen conclude that $c_1 (L) = -1$.  The discrepancy can be explained using the footnote regarding signs in Section~\ref{g=1} of the present article.}  For related work in the algebraic category, see Drezet--Narasimhan~\cite{DN}.

We remark that it would be interesting to extend this degree calculation to other simply connected compact Lie groups.

\vspace{.2in}

\noindent {\bf Acknowledgements:} The second author thanks Simon Donaldson for suggesting that the results of \cite{R} should be connected to Goldman's symplectic form.  Additionally, we thank Jacques Hurtubise for asking about the higher genus case considered in Section~\ref{g>1}, and the anonymous referees for helping to improve the exposition.

\section{The Chern-Simons line bundle}
Let $G$ be a simply connected, compact Lie group, equipped with a chosen faithful representation into $\textrm{GL}(n, \CC)$, and let  $\mathfrak{g}$ be the Lie algebra of $G$ (viewed as a subalgebra of $\mathfrak{gl}(n, \CC)$).
The space of connections on the trivial $G$--bundle over $\Sigma$ will be denoted by $\calA = \Omega^1 (\Sigma, \mathfrak{g})$, and the gauge group of this bundle will be denoted by  $\mathcal{G} = C^\infty (\Sigma,G)$.  

The line bundle from \cite{RSW} 
is defined using the Chern-Simons cocycle (\cite{RSW}, p. 411) 
$\Theta:  \calA  \times \calG \to \CC$ defined by
$$ \Theta(A, g) = \exp  i  (CS ({\bf A}^{\bf g}) -  CS({\bf A})  ). 
$$ 
The Chern-Simons functional $CS({\bf A})$ is defined by
$$CS({\bf A}) = \frac{1}{4 \pi} \int_N {\rm Trace} ({\bf A} d{\bf A}
 + \frac{2}{3} {\bf A}^3) $$
where $N$ is a 3-manifold with boundary $\Sigma$
and $g \in \mathcal{G} = C^\infty (\Sigma,G).$
We have chosen extensions $\bf A$ and $\bf g$ of $A$  and $g$ (respectively) over the bounding
3-manifold $N$ (the existence of $\bf g$ relies on simple connectivity of $G$).
It is shown in \cite{RSW} that the Chern-Simons cocycle $\Theta(A, g)$ is independent
of the choice of these extensions.
We define a line bundle $\lineb$ over $\mathcal{A_F}/\gauge$ as a 
$\gauge$-equivariant bundle over the space of flat
connections $\mathcal{A}_F$, where $g \in \gauge$ acts
on $\mathcal{A} \times \CC$ by 
$$ g: (A,z) \mapsto (A^g, \Theta(A, g) z). $$
The definition of $\lineb$ is
$$ \lineb = \mathcal{A}_F \times_{\mathcal{G}} \CC. $$

The symplectic form $\hat{\Omega}$ on $\mathcal{A}$ is defined by 
(see \cite{RSW}, p. 412):
\beq \label{sym}\hat{ \Omega}(a,b) = \frac{i}{2\pi} \int_\Sigma {\rm Trace} (a \wedge b) \eeq
for $a,b \in \Omega^1(\Sigma, \mathfrak{g})$.
Notice that on the affine space $\mathcal{A}$, the symplectic form
is a constant quadratic form;
 it does not depend on choosing a point in $\mathcal{A}. $

\section{Degree of the Chern-Simons line bundle in genus 1}$\label{g=1}$

Let $N$ be a three-manifold with boundary $\Sigma$.

The symplectic form on $\mathcal{A}$ from (\ref{sym}) descends to a 2--form
$\Omega$ on $\mathcal{A}_F/\gauge$ (the space of flat connections), which is 
symplectic when restricted to the subspace $\mathcal{A}_F^s \subset \mathcal{A}_F$ 
of irreducible flat connections.  

The authors of \cite{RSW} exhibit a unitary connection $\hat{\omega}$  on
the prequantum line bundle over $\mathcal{A}_F$:
\beq \label{hatom}\hat{\omega}(a)= \frac{i}{4 \pi} \int_\Sigma {\rm Trace} (A \wedge a) \eeq
whose curvature is $\hat{\Omega}$. 
This is done on p. 412 of \cite{RSW}.
The proof uses the fact that the derivative of the Chern-Simons function is
$$dCS_A(v) = \frac{1}{4 \pi} \left( \int_{N} 2 {\rm Trace} (v \wedge F_A) - \int_\Sigma
{\rm Trace} (A\wedge v)\right)$$
for $v \in T_A \mathcal{A} = \mathcal{A} = \Omega^1(N, \mathfrak{g})$.
This follows from a straightforward calculation using Stokes' theorem.
The above expression restricts on $\mathcal{A}_F$ to 
$$dCS_A(v) =  -\frac{1}{4 \pi} \int_\Sigma {\rm Trace} (A\wedge v)  = i \hat{\omega}(v) $$
 (recalling (\ref{hatom})).
It is shown on p. 412 of \cite{RSW} (second paragraph) that
$\hat{\omega}$ is the pullback of a connection $\omega$  on 
$\mathcal{A}_F \times_{\gauge} \CC.$
This is demonstrated by introducing a vertical vector field
$Y$ for the action of $\gauge$, and 
showing that $$i_Y \hat{\omega} = L_Y \hat{\omega}  = 0  $$
so $\hat{\omega}$ is  basic, and therefore descends
to a 1-form on $\mathcal{A}_F/\gauge$. 

For the rest of the section,  we restrict to $G=SU(2)$.
Let $x, y$ be the flat coordinates on the genus $1$
 surface (see the proof of Lemma \ref{lemone} for more details).
Inside the space $\cA$ we can consider the space $\mathcal{W}$ 
of all connections of the form $a\, dx
+ b \,dy$ where $a, b  \in{\rm  Lie}(T)$
and $T = \left\{ {\rm diag } ( \lambda, \lambda^{-1}) \,:\, \lambda \in S^1\right\}$ is the diagonal maximal torus of $SU(2)$.  Note that $ {\rm  Lie}(T) = \{xX \,:\, x\in \RR\}$, where $X = {\rm diag } (i,-i) \in su(2).$

Now $\mathcal{W}$ is a subspace of $\cA$ so the bundle $\lineb$
 restricts to $\mathcal{W}$ as a bundle with
connection.  This bundle
is invariant under that part of the gauge group that preserves $\mathcal{W}$.
  This consists of $ (\ZZ
\times \ZZ) \ltimes \ZZ_2.$
Here $(m,n) \in \ZZ \times \ZZ$ is identified with 
the gauge transformation $ (e^{ix}, e^{iy}) \mapsto 
e^{imx}  e^{iny} $, and
$ \ZZ_2 = \{\pm 1\}$ 
is the Weyl group of $SU(2)$.

Taking the quotient by $\ZZ \times \ZZ$ we get a bundle $\lineb'$ on  $T \times T = \mathcal{W}/ \left(\ZZ \times \ZZ\right)$
 with a connection $\omega'$.  We will show, via a direct computation (Lemma~\ref{lemone}), that the curvature $\Omega'$ of this connection has integral
equal to $-4 \pi i $, and using Chern-Weil theory we will be able to conclude that $\lineb'$ has degree 2, while $\lineb$ has degree 1.

The computation goes as follows.

\newcommand{\hatOm}{\hat{\Omega}}

\begin{lemma} \label{lemone}
We have  $$\int_{T\cross T} \Omega'   =  - 4 \pi i . $$
\end{lemma} 

\begin{proof} 
As above, let
$X = {\rm diag } (i,-i) \in su(2).$
Then ${\rm Trace}(X^2) = -2. $  

Let $\Gamma$ be a fundamental domain for the action of 
$\ZZ \times \ZZ $ on ${\rm  Lie}(T) \oplus {\rm  Lie}(T)$.
Parametrize $\Gamma$ by $(x,y)\in [0, 2\pi] \times[0, 2\pi]$. 
Under the exponential map $\exp: {\rm Lie}(T) \to T$,
the vector fields
$\frac{\partial}{\partial x}$ and $\frac{\partial}{\partial y}$ on 
$\Gamma$ 
are  identified with the constant vector field 
$X$ on $T$,  so with the above formula (\ref{sym}) for $\hat{\Omega}$ we have that
$$\Omega'_{(x,y)} := \Omega'_{(x,y)}(\frac{\partial}{\partial x},\frac{\partial}{\partial y}) 
= \hat{\Omega}_{(x,y)} (\frac{\partial}{\partial x},\frac{\partial}{\partial y})  
= \frac{i}{2 \pi} \int_0^{2\pi}\int_0^{2\pi} {\rm Trace}(X^2) dx dy $$
$$ = -4 \pi i. $$ 
The value of $\Omega'_{(x,y)}$ is independent of $x$ and $y$, so the integral of $\Omega'$ over $T\cross T$ (integrating using an area form of total area 1) is also $-4\pi i$.
\end{proof}

\noindent {\bf Proof of Theorem \ref{thm1} in genus 1.} 
By Lemma~\ref{lemone}, the cohomology class $[\Omega']\in H^2 (T\cross T; \mathbb{R})$ associated to $\Omega'$ is $-4\pi i \alpha$, where 
$\alpha = \frac{1}{(2\pi)^2} [dx\wedge dy] \in H^2 (T\cross T; \mathbb{R})$ denotes the fundamental class.
By Chern--Weil theory\footnote{In~\cite[Appendix C]{MS}, the Chern--Weil formula for characteristic classes is stated without signs, because they use a version of Fubini's Theorem with signs~\cite[p. 304]{MS}.  Since we have integrated using the usual version of Fubini's Theorem, we need a sign in our formula for $c_1 (\lineb')$.}, we have 
$$[\Omega'] = -2\pi i \,c_1 (\lineb'),$$
so $c_1 (\lineb') =  \frac{2}{(2\pi)^2} [dx\wedge dy]= 2 \alpha$, and $\lineb'$ has degree two.

The generator of $W = \mathbb{Z}/2\mathbb{Z}$ acts on $T\cross T$ by complex conjugation on each factor, inducing 
a quotient map $f\co T\cross T \to (T\cross T)/W$, and we have a homeomorphism $(T\cross T)/W \cong S^2$.
Our bundle $\lineb'$ is $\ZZ_2$--equivariant, and it descends to the bundle $\lineb$ on $(T \times T)/ W$ (note here that if $(z,w)\in T\cross T$ is fixed by $W$, then the action of $W$ on the fiber of $\lineb'$ over $(z,w)$ is \emph{trivial}: this action is defined in terms of the cocycle $\Theta (A, g)$, which is zero whenever $g$ fixes $A$).  
So $f^*(\lineb) = \lineb'$, and hence 
$\deg (\lineb)  
=\frac{1}{\deg (f)} \deg (\lineb')$.  An elementary calculation (e.g. using a $\ZZ_2$--equivariant  CW complex structure on $T\cross T$) shows that $\deg(f) = 2$, completing the proof.
$\hfill \Box$

\vspace{.25in}

  The key point is that we have
computed the degree on $T\cross T$, which has a canonical smooth manifold structure, so we can use Chern-Weil theory.
The proof of Theorem~\ref{thm1} in higher genus will be given in Section~\ref{g>1}.

\section{The conjecture of Lawton and Ramras on the Chern-Simons line bundle}$\label{conj}$

Let $\Sigma$ be a closed oriented 2-manifold of genus $g>0$.  Let  $\calG_{SU(n)} = \calG_{SU(n)} (\Sigma)$ and $\calG_{U(n)} = \calG_{U(n)} (\Sigma)$ denote the gauge groups of the trivial $SU(n)$ and $U(n)$--bundles on $\Sigma$ (respectively), and let $\calA^{SU(n)}_F(\Sigma)$ and $\calA^{U(n)}_F(\Sigma)$ denote the spaces of flat $SU(n)$-- and $U(n)$--connections on these bundles.
Define
$$\mathcal{M}_U (\Sigma)= \colim_n  \calA^{U(n)}_F(\Sigma)/\calG_{U(n)} \textrm{ and } \mathcal{M}_{SU} (\Sigma) = \colim_n  \calA^{SU(n)}_F(\Sigma)/\calG_{SU(n)}.$$  
We refer to these as the \emph{stable moduli spaces} of flat unitary (or special unitary)  connections over $\Sigma$.  Let $\lineb_n  \to \calA^{SU(n)}_F(\Sigma)/\calG_{SU(n)}$ denote the prequantum line bundle.  As $n$ varies, these bundles are compatible with the inclusions $SU(n) \injects SU(n+1)$, and hence induce a line bundle $\lineb_\infty \to \MSU (\Sigma)$, which we call the stable prequantum line bundle.

The homotopy types of the stable moduli spaces were determined in \cite{LR, R}:
\begin{equation}\label{LRR}\MU (\Sigma) \heq  \CC P^\infty \times (S^1)^{2g}, \,\,\,\, \MSU(\Sigma) \heq \CC P^\infty.\end{equation}
These results are computational, and rely on the uniqueness of Eilenberg--MacLane spaces; that is, no explicit homotopy equivalences between these spaces have been constructed.
Here we  offer bundle-theoretic descriptions of these homotopy equivalences.

\begin{remit} $\label{CW}$ The proof that $\MU (\Sigma)$ and $\MSU(\Sigma)$
have the homotopy types stated above relies on an independent result showing that these spaces
have the homotopy types of CW complexes.  In the unitary case, this is proven in~\cite[Lemma 5.7]{R}, and the same argument works in the special unitary case.
\end{remit}

\begin{theorem}\label{SU} The classifying map
$$\MSU (\Sigma) \maps \CC P^\infty$$
for the stable prequantum line bundle $\lineb_\infty$  
is a homotopy equivalence.
\end{theorem}
\begin{proof}
Writing $\Sigma = \Sigma' \# T$, where $T$ is a torus,
the quotient map induces a homeomorphism $\Sigma \to \Sigma/\Sigma' \cong T$. Together with the inclusions $SU(2)\hookrightarrow SU(n)$, this induces a map  
$$\calA^{SU(2)}_F(T)/\calG_{SU(2)} (T)\maps \MSU (\Sigma),$$
which induces an isomorphism on $H^2 ( -; \ZZ)$ by \cite[Theorem 5.3]{LR}.
We have shown in the previous section that the classifying map for the bundle $\lineb_2  \to \calA^{SU(2)}_F(T)/\calG_{SU(2)} (T)$
induces an isomorphism on $H^2(- ; \bbZ)$.  Since the classifying map for $\lineb_\infty$ restricts to a classifying map for $\lineb_2$, 
we find that the classifying map for $\lineb_\infty$ must also induce an isomorphism on $H^2 (-; \bbZ)$.  But up to homotopy, this is a self-map of $\CC P^\infty$, and any self-map of $\CC P^\infty$ that induces an isomorphism on 
$H^2 (-; \bbZ)$ is a homotopy equivalence.
\end{proof}

We now turn to the unitary case.  Fix generators $\alpha_i$, $\beta_i$ ($i=1, \ldots, g$) for $\pi_1 (\Sigma)$ (with $\prod_i [\alpha_i, \beta_i] = 1$).  The determinant map
$$\det \co \MU(\Sigma) \to (S^1)^{2g}$$
is defined by 
$$[A] \mapsto (\det(\rho_A (\alpha_1)), \det(\rho_A (\beta_1)), \ldots, \det(\rho_A (\alpha_g)), \det(\rho_A (\beta_g))),$$
where $A\in \calA^{U(n)}_F(\Sigma)$ and $\rho_A \co \pi_1 (\Sigma) \to U(n)$ is its holonomy representation.

\begin{corollary} There is a line bundle $\lineb^U \to \MU(\Sigma)$
that restricts to $\lineb_\infty  \to \MSU(\Sigma)$, and if $\alpha$ is a classifying map for $\lineb^U$, then
the map
$$\MU(\Sigma) \xmaps{(\det, \alpha)}   (S^1)^{2g}\times \CC P^\infty$$
is a homotopy equivalence.
\end{corollary}
\begin{proof}   
 Let $f\co \MSU\srm{\heq} \CC P^\infty$ be a classifying map for  $\lineb_\infty$, and choose a homotopy equivalence $\phi\co \MU(\Sigma) \srm{\heq} \CC P^\infty \cross (S^1)^{2g}$, as in (\ref{LRR}).  Let 
$\CC P^\infty\cross (S^1)^{2g}\srm{p_1} \CC P^\infty$ be the the projection, and let 
 $i\co \MSU(\Sigma) \injects \MU(\Sigma)$ be the inclusion. By   \cite[Theorem 5.3]{LR}, $i$
 induces an isomorphism on $\pi_2$, 
so the composite
$$i_1 := p_1 \circ \phi \circ i \co \MSU(\Sigma) \maps  \CC P^\infty$$ 
induces an isomorphism on $\pi_2$, hence is a homotopy equivalence.  
Define 
$$\alpha :=  f \circ i_1^{-1} \circ p_1 \circ \phi \co \MU(\Sigma) \to \CC P^\infty,$$ where $i_1^{-1}$ is a homotopy inverse to $i_1$.  Then we have a homotopy commutative diagram
$$ \xymatrix{\MSU(\Sigma)  \ar@{^{(}->}[rr]^-i \ar[dr]_-f^-\heq & & \MU(\Sigma)  \ar[dl]^\alpha \\ & \CC P^\infty}$$
and we define $\lineb^U$ to be the pullback, under $\alpha$, of the universal bundle over $\CC P^\infty$.
It remains to show that $(\det, \alpha)$ is a homotopy equivalence. Since $\det$ is split by the inclusion of $(S^1)^{2g} \cong \Hom(\pi_1 (\Sigma), U(1))/U(1)$ into $\MU(\Sigma)$,
we see that on fundamental groups, $\det_*$ is a surjection between free abelian groups of rank $2g$, hence an isomorphism.  Since $\alpha_*$ is an isomorphism on $\pi_2$, the result follows from the Whitehead Theorem (and 
Remark~\ref{CW})   
\end{proof}

\section{The degree of the line bundle in higher genus}$\label{g>1}$

Let   $\lineb_g$ denote the prequantum line bundle on the moduli space $\calA_F^{SU(2)}(\Sigma^g)/\calG$ of flat connections on a trivial $SU(2)$--bundle over the genus $g$ surface $\Sigma^g$.
We now show that $\lineb_g$ has degree $1$ for every genus $g$ surface ($g>0$), not just $g=1$ (thereby completing the proof of Theorem~\ref{thm1}).  This statement is meaningful, since we have:

\begin{lemma}$\label{betti}$ For any $g\geq 1$, we have $H^2 (\calA_F^{SU(2)}(\Sigma^g)/\calG; \bbZ) \cong \bbZ$.
\end{lemma}  
\begin{proof} 
In \cite{LR}, it was proven that $\calA_F^{SU(2)}(\Sigma^g)/\calG$ is simply connected (and a second proof of this fact was given in \cite{BLR}).  Now, triviality of $H_1 (\calA_F^{SU(2)}(\Sigma^g)/\calG; \bbZ)$ implies that $H^2 (\calA_F^{SU(2)}(\Sigma^g)/\calG; \bbZ)$ is torsion-free $($by the Universal Coefficient Theorem$)$, and a simple direct analysis of the Poincar\'{e} polynomial of $\calA_F^{SU(2)}(\Sigma^g)/\calG$, as determined by Cappell--Lee--Miller~\cite{CLM}, shows that $H^2 (\calA_F^{SU(2)}(\Sigma^g)/\calG; \bbZ)$ has rank 1.  
\end{proof}

A map $f\co \Sigma^g \to \Sigma^h$ 
induces a map 
$$f^\#\co \calA_F^{SU(2)}(\Sigma^h)/\calG \to \calA_F^{SU(2)}(\Sigma^g)/\calG,$$ 
and as noted in  \cite[Remark 3, p. 412]{RSW}, if $f$ has degree 1 then
 $(f^\#)^*(\lineb_g) = \lineb_h$.  This implies that 
 $$(f^\#)^* (c_1(\lineb_g)) = c_1 (\lineb_h),$$
  and taking $h = 1$ we find that 
  $(f^\#)^* (c_1(\lineb_g)) = c_1 (\lineb_1) = 1$  
  (by the result in Section~\ref{g=1}).  Now Lemma~\ref{betti} implies that $c_1(\lineb_g)$ is a generator of $H^2 (\calA_F^{SU(2)}(\Sigma^g)/\calG; \bbZ)$.  This completes the proof of Theorem~\ref{thm1}. $\hfill \Box$

 \begin{remit} The results in \cite{LR} in fact show that the map
 $$H^2 (\MSU(\Sigma^g); \bbZ) \maps H^2 (\MSU(\Sigma^1); \bbZ)$$
 is determined by
 $$f^*\co H^2 (\Sigma^1) \to H^2 (\Sigma^g).$$  Thus a choice of generator in 
 $H^2 (\calA_F^{SU(2)}(\Sigma^1)/\calG; \bbZ)$ and a choice of orientations on $\Sigma^1$ and $\Sigma^g$ give a choice of generator in 
$H^2 (\calA_F^{SU(2)}(\Sigma^g)/\calG; \bbZ)$, and the above discussion shows that this generator coincides with $c_1 (\lineb_g)$.
 \end{remit}

\end{document}